\documentclass[20pt]{article}

\usepackage{amssymb}
\usepackage{amsfonts}
\usepackage{amsmath}
\usepackage{mathrsfs}
\usepackage{amsthm}

\newtheorem{theorem}{Theorem}[section]
\newtheorem{lemma}[theorem]{Lemma}
\newtheorem{proposition}[theorem]{Proposition}

\newtheorem{definition}[theorem]{Definition}

\newtheorem{conjecture}[theorem]{Conjecture}

\newtheorem{remark}{Remark}
 \makeatletter
    
\input epsf
\usepackage{graphics}
\usepackage{epsfig}
\usepackage{mathrsfs}
 
\begin{document}

\begin{center}
 {\bf  REFINEMENT EQUATIONS AND SPLINE FUNCTIONS}
\end{center}

\begin{center}
{Art\={u}ras Dubickas}

\vspace{0.3cm} {\small Department of Mathematics and Informatics,
Vilnius University,

Naugarduko 24, Vilnius LT-03225, Lithuania}

{Email: arturas.dubickas@mif.vu.lt}

 {\small and}

{\small Institute of Mathematics and Informatics,

Akademijos 4, Vilnius LT-08663, Lithuania}

\end{center}

\begin{center}
 {Zhiqiang Xu }

\vspace{0.3cm}

{\small Institute of Computational Math. and Sci. and Eng.
Computing,

 Academy of Mathematics and System Sciences,

 Chinese
 Academy of Sciences, Beijing, 100080 China}

{Email: xuzq@lsec.cc.ac.cn}
\end{center}



In this paper, we exploit the relation between the regularity of
refinable functions with non-integer dilations  and the
distribution of powers of a fixed number modulo $1$, and show the
nonexistence of a non-trivial ${\bf C}^{\infty}$ solution of the
refinement equation with non-integer dilations.  Using this, we
extend the results on the refinable splines with non-integer
dilations and construct a counterexample to some conjecture
concerning the refinable splines with non-integer dilations.
Finally, we study the box splines satisfying the refinement
equation with non-integer dilation and translations. Our study
involves techniques from number theory and harmonic analysis.

\vspace{0,5cm} \noindent {\bf Key Words:} {Refinement equation,
multivariate spline,  Fourier transform, distribution modulo 1.}

\vspace{0.5cm} \noindent {\bf Mathematics Subject Classification
(2000):} {41A15, 05A17, 11J71.}

\baselineskip 20pt

\section{Introduction}

The {\it refinement  equation} is a functional equation of the
form
\begin{equation}\label{eq:refinable}
f(x)=\sum_{j=0}^Nc_j f(\lambda x-d_j),
\end{equation}
where $\lambda>1$ and all the $c_j, d_j$ are real numbers. For the
refinement equation (\ref{eq:refinable}), the value $\lambda$ is
called a {\it dilation}, whereas the numbers  $\{d_j\}$ are
referred to as {\it translations}. Throughout this paper, we
suppose that $d_0<d_1<\cdots<d_N$ and define the Fourier transform
of $f(x)$ by the formulae
$$
\widehat{f}(w)=\int_{-\infty}^\infty f(x)e^{-2\pi i w x}dx.
$$
Taking the Fourier transform of both sides of (\ref{eq:refinable})
we obtain
\begin{equation}\label{eq:refour}
\widehat{f}(w)=H(\lambda^{-1}w) \widehat{f}(\lambda^{-1}w),
\end{equation}
where $H(w)=\lambda^{-1}\sum_{j=0}^Nc_je^{-2\pi id_jw}$ is called
the {\it mask polynomial} of the refinement  equation. Setting
$w=0$ into (\ref{eq:refour}), we obtain $\sum_{j=0}^Nc_j=\lambda$
provided that $\widehat{f}(0)\neq 0$. For simplicity, we shall
call the function $f(x)$ satisfying (\ref{eq:refinable}) with
$\widehat{f}(0)=1$ a {\it $\lambda$-refinable function} with
translations $\{d_j \>|\> 0\leq j \leq N\}$. It plays a
fundamental role in the construction of compactly supported
wavelets and in the study of subdivision schemes in CAGD
(\cite{cdm,D2}).

The existence and regularity of the refinable function are of some
interest. These questions were studied on several occasions. In
\cite{D1}, Daubechies and Lagarias showed that up to a scalar
multiple the refinement equation (\ref{eq:refinable}) has a unique
distribution solution $f$ satisfying ${\rm supp} f\subset
[d_0(\lambda-1)^{-1},d_N(\lambda-1)^{-1}]$. Moreover, they also
showed  the nonexistence of a ${\bf C}^{\infty}$-refinable
function with compact support in one dimension when  the number
$\lambda$ and all the $d_j$ are integers. In \cite{cdm}, Cavaretta
et al. extended this result to higher dimensions by a matrix
method. When $\lambda$ is non-integer, `the regularity question
becomes more complicated and perhaps more interesting from the
viewpoint of pure analysis' \cite{dai}. Moreover, the refinable
functions with non-integer dilations play an important role in the
construction of wavelets with non-integer dilations \cite{ausch}.
Hence, it has attracted a considerable attention. For example, the
regularity of Bernoulli convolutions, which are solutions to
$$
f(x)=\frac{\lambda}{2}f(\lambda x)+\frac{\lambda}{2}f(\lambda
x-1),
$$
was already studied in \cite{b1,b2,b3,b4}. In general, one
characterizes the regularity of $f(x)$ by considering the decay of
$\widehat{f}(w)$.
 In \cite{dai}, Dai et al. considered the uniform decay of the
Fourier transform of the refinable functions with non-integer
$\lambda$. In particular, they give an elegant answer to the
following question: {\it for any given dilation factor $\lambda>1$
and a positive integer $k$, can one construct a
$\lambda$-refinable function $f\in {\bf C}^k$?} In this paper, we
reverse the question: {\it can one find a dilation factor
$\lambda>1$ such that there exists a compactly supported
$\lambda$-refinable function $f\in {\bf C}^{\infty}$?}

This question is interesting for two reasons. First, when
$\lambda$ and all the $d_j$ are integers, the regularity of
refinable functions is related to the spectrum of a certain matrix
that is constructed by the refinement coefficients $c_j$. In this
case, a negative answer to this question was given in \cite{D1}.
However, the matrix can only be constructed in the case when
$\lambda$ and all the $d_j$ are integers. Therefore, it is not
clear whether there exists a non-integer refinable function which
is compactly supported and belongs to ${\bf C}^\infty$. Second, as
pointed out in Section 3, the regularity of refinable functions is
closely related to the distribution of powers of a fixed number
modulo 1, which is a classical problem in number theory. Hence,
overlooking the regularity of refinable functions with non-integer
dilations one will miss out on the beautiful connection between
analysis and number theory.

 In this
paper, combining the tools of number theory with some results of
harmonic analysis, we extend the result of Daubechies and Lagarias
to the general case, thus giving a negative answer to the above
question when  all the $d_j$ are rational numbers and $\lambda>1$
is a real number.

\begin{theorem}\label{th:noninf}
Let $d_j\in {\bf Q}$ and $\lambda>1$. Then the refinement equation
$$
f(x)=\sum_{j=0}^N c_j f(\lambda x-d_j),
$$
has only the trivial compactly supported ${\bf C}^{\infty}$
solution, i.e., $f\equiv 0$.
\end{theorem}

The proof of Theorem~\ref{th:noninf} is based on a purely
number-theoretic statement concerning the distribution of powers
of a fixed number modulo $1$ (see Theorem~\ref{th:number} below).

\begin{remark}
It should be noted that any translate of a refinable function is
refinable. If $f(x)$ satisfies (\ref{eq:refinable}) then
$g(x)=f(x-b/(\lambda-1))$ satisfies the refinement equation
$$
g(x)=\sum_{j=0}^N c_j g(\lambda x-d_j+b),
$$
which has the same dilation, but a different translation set
$\{d_j-b \>|\> 0\leq j \leq N\}$.

\end{remark}

 In Section 2, using
Theorem~\ref{th:noninf}, we extend the results of \cite{daispline}
concerning the refinable splines with non-integer dilations. We
also construct a counterexample to the conjecture about the
refinable spline with non-integer translations. The box spline
$B(x|M)$ associated with the $s\times n$ matrix $M$ is considered
as multivariate generalization of the univariate  B spline and
have become `one of the most dramatic success of multivariate
splines' \cite{devore}.  It is well-known that the box spline is a
$\lambda$-refinable function for any positive integer $\lambda>1$.
Hence, box splines play an important role in the subdivision
algorithm and wavelets. Moreover, one is also interested in the
question when a refinable distribution is essentially a box spline
\cite{goodman,sun}. However, very little is known about  box
splines satisfying the refinement equations with non-integer
dilations. In Section 2, we give a characterization of refinable
box splines with non-integer dilations.

The proof of Theorem~\ref{th:noninf}  will be given in Section 3.
Other proofs concerning refinable splines will be given in
Sections 4 and 5.

\section{Refinable Spline Functions}

In this section, we study one of the most interesting refinable
functions, so-called spline function. Firstly, let us recall the
definition of spline functions, B splines and box   splines.
\begin{definition}
The {\it spline} function $f(x)$ is a piecewise polynomial
function. More precisely, there are some points
$-\infty=x_0<x_1<\cdots <x_M<x_{M+1}=+\infty$ and polynomials
$P_j(x)$ such that $f(x)=P_j(x)$ for $x\in [x_{j-1},x_j)$ for each
 $j=1,\ldots,M+1$.
\end{definition}
 We call the points
$x_j$, $j=1,\ldots,M$, the {\it knots} of $f(x)$ and
max$_j\{\deg(P_j)\}$ the {\it degree} of $f(x)$. Splines are
widely used in the approximation theory and in computer aided
geometric design. The spline functions satisfying a refinement
equation are the most useful ones. A special class of refinable
splines are so-called  {\it cardinal B splines} which are defined
by induction as
$$
B_0(x)=\begin{cases}
1 & \text{if} \>\> x\in[0,1),\\
0 & \text {otherwise},
\end{cases}
$$
and for $k\geq 1$
$$
B_k=B_{k-1}*B_0,
$$
where $*$ denotes the operation of convolution. By the definition of
cardinal B splines, the Fourier transform of B splines is given by
the formulae
$$
\widehat{B}_k(w)=\left(\frac{1-exp(-2\pi iw)}{2\pi i
w}\right)^{k+1}.
$$

A simple observation is that $B_0$ is $m$-refinable for any integer
$m>1$ and satisfies
$$
B_0(x)=\sum_{j=0}^{m-1} B_0(mx-j).
$$
Since $B_k=B_0*B_0*\cdots *B_0$, the B spline $B_k$ is also
$m$-refinable for any integer $m>1$.

The {\it box spline } $B(x|M)$ is defined using Fourier transform
by the formulae
\begin{equation}\label{eq:boxcube}
\widehat{B}(\xi |M)=\prod_{j=1}^n\frac{1-exp(-2\pi
i\xi^Tm_j)}{2\pi i\xi^Tm_j},\ \ \ \ \ \xi \in {\bf C}^s,
\end{equation}
where $M=(m_1,\ldots,m_n)$ is an $s\times n$ real matrix of (full)
rank $s$. In particular, for $M=(1,\ldots,1)\in {\bf Z}^{k+1}$,
the box spline $B(x|M)$ is reduced to the univariate B spline
$B_k$.

Next, we turn to the general refinable spline function. In
\cite{lawton}, Lawton et al. gave the characterization of
compactly supported refinable univariate splines $f(x)$ with the
additional assumption that the  dilation factor is an integer and
all the translations are integers. They proved the following
theorem:
\begin{theorem}{\rm (\cite{lawton})}\label{th:lawton}
Suppose that $f(x)$ is a compactly supported spline function of
degree $d$. Then $f(x)$ satisfies the  refinement equation
$$
f(x)=\sum_{j=0}^N c_j f(mx-d_j),\,\, m>1,\, d_j\in {\bf Z},
$$
if and only if $f(x)=\sum_{n=0}^Kp_n B_d(x-n-d_0/(m-1))$ for some
$K\geq 0$ and $\{p_n\}$  such that the polynomial $Q(z)=
(z-1)^{d+1}\sum_{n=0}^K p_nz^n$ satisfies $Q(z)|Q(z^m)$.
\end{theorem}

Here and below, $Q(z)|Q(z^m)$ means $Q(z)$ divides $Q(z^m),$ namely,
that the quotient $Q(z^m)/Q(z)$ is a polynomial. (In case the
notation $a|b$ is used for integers $a$ and $b,$ it means that the
quotient $b/a$ is an integer.)
\begin{remark}
Theorem \ref{th:lawton} was extended to higher dimensions by Sun
\cite{sun}. Moreover, in \cite{lawton}, the authors also gave a
characterization of compactly supported univariate refinable
splines whose shifts form a Riesz sequence. This was generalized
to higher dimensions in \cite{guan}. In \cite{dai2}, the authors
gave complete characterization of the structure of refinable
splines. In \cite{goodman}, Goodman gave a review on refinable
splines (including refinable vector splines).
\end{remark}

However, as pointed out in \cite{daispline}, refinable splines do
not have to have integer dilations or integer translations. In
this case, Dai et al. proved the following theorem:

\begin{theorem}{\rm (\cite{daispline})}\label{th:respline}
Suppose that $f(x)$ is a compactly supported spline satisfying the
refinement equation
\begin{equation}\label{eq:daire}
f(x)=\sum_{j=0}^Nc_jf(\lambda x-d_j),\,\,\,\, \sum_{j=0}^N
c_j=\lambda
\end{equation}
such that  $\lambda\in {\bf R}$ and $d_j\in {\bf Z}$. Then

(A) There exists an integer $l>0$ such that $\lambda^l\in {\bf
Z}.$

(B) Let $k$ be the smallest positive integer such that
$\lambda^k\in {\bf Z}.$ Then the compactly supported distribution
solution $\phi(x)$ of the refinement equation
\begin{equation}\label{eq:daith}
\phi(x)=\sum_{j=0}^nc_j\lambda^{k-1}\phi(\lambda^kx-d_j)
\end{equation}
is a spline.

(C) There exists a constant $\alpha$ such that the spline $f(x)$
is expressible as
\begin{equation}\label{eq:dais}
f(x)=\alpha \phi(x)*\phi(\lambda^{-1}x)*\cdots
*\phi(\lambda^{-(k-1)}x),
\end{equation}
where $\phi(x)$ is the spline given by (\ref{eq:daith}).

Conversely, if the refinement equation (\ref{eq:daire}) satisfies
(A) and (B) then the compactly supported distribution solution $f$
is a spline given in (\ref{eq:dais}).
\end{theorem}

By Theorem \ref{th:noninf}, we can prove Theorem \ref{th:respline}
under weaker conditions. A compactly supported function on ${\bf
R}$ is {\it piecewise smooth} if there exist an integer $M$ and
some real numbers $a_1<a_2<\cdots <a_M$ such that $f\in {\bf
C}^\infty(a_{j},a_{j+1})$ for $j=1,2,\dots, M-1$ and ${\rm supp}
f\subset [a_1,a_M]$.

\begin{theorem}\label{th:gene}
Let $f$ be a piecewise smooth function with compact support
satisfying the refinement  equation
\begin{equation}
f(x)=\sum_{j=0}^Nc_j f(\lambda x-d_j),\ \ \ \ \
\sum_{j=0}^Nc_j=\lambda,
\end{equation}
where $\lambda>1$ and $d_j\in {\bf Z}$. Then $f(x)$ is a spline
function. Hence, (A), (B) and (C) in Theorem  \ref{th:respline}
hold.
\end{theorem}

To deal with non-integer translations, in \cite{daispline}, the
authors raised the following conjecture. {\it Suppose that $f(x)$
is a $\lambda$-refinable spline that is $\lambda$-indecomposable.
Then the translation set for $f$ must be contained in a lattice,
i.e., a set of the form $a{\bf Z}+b$ for some $a\neq 0$}. Here, we
call a $\lambda$-refinable spline $f(x)$ {\it
$\lambda$-indecomposable} if it cannot be written as the
convolution of two $\lambda$-refinable splines. As stated in
\cite{daispline}, if this conjecture is true, then one can
classify all refinable splines by Theorem \ref{th:respline}.
However, we construct a counterexample to the conjecture.

 Consider the spline function $f(x)=B(x|(1,\sqrt{5/2}))$.
 Using the Fourier transform
$$
\widehat{f}(w)=\frac{1-e^{-2\pi i w}}{2\pi i w} \frac{1-e^{-2\pi i
\sqrt{{5}/{2}}  w}}{2\pi i\sqrt{{5}/{2}} w},
$$
we have $H(w)=\widehat{f}(\sqrt{10} w)/\widehat{f}(w)=\frac{1}{10}\,
\sum_{j=0}^9 e^{-2\pi i d_j w}$, where the numbers $d_0,\ldots,d_9$
are equal to
$0,1,2,3,4,5/\sqrt{10},1+5/\sqrt{10},2+5/\sqrt{10},3+5/\sqrt{10},4+5/\sqrt{10}$,
respectively. Hence, the spline $f(x)$ is $\sqrt{10}$-refinable and
satisfies the refinement equation
$$
f(x)=\frac{1}{10} (\sum_{j=0}^9f(\sqrt{10}x-d_j))
$$
with the translation set
$\{0,1,2,3,4,5/\sqrt{10},1+5/\sqrt{10},2+5/\sqrt{10},3+5/\sqrt{10},4+5/\sqrt{10}\}$,
which is not contained in a lattice. Moreover, we have the
following proposition:

 \begin{proposition}\label{pr:exam}
The univariate  box spline $B(x|(1,\sqrt{5/2}))$ is
 $\sqrt{10}$-indecom\-posable.
 \end{proposition}

The proof of this proposition is non-trivial. It is postponed to
Section 4.

Motivated by this counterexample, we shall study the refinable Box
spline function with non-integer dilation and translations. It is
not only helpful in the study of general refinable splines with
non-integer translations, but also useful in understanding of the
box splines associated with non-integer matrixes.

It is well-known that the box spline $B(x|M)$, with an additional
assumption that $M$ is an $s\times n$ integer matrix, satisfies
some higher dimension refinement equation with an integer $m>1$,
i.e.,
\begin{equation}
B(x|M)=\sum_{{ j}\in {\bf Z}^s}c_{ j}^{(m)} B(mx-j|M),
\end{equation}
where   $c_{ j}^{(m)}=m^{s-n}\#\{\alpha \in {\bf Z}^n|M\alpha = m
{ j} \}$ (see \cite{deboorbook}). For this reason, box splines are
widely used in computer aided geometry design and wavelets.

The  box splines associated with the non-integer matrix are also
of interest. In \cite{jia1, zhou}, the authors discuss linearly
independent integral lattice translates of the box splines
associated with a non-integer matrix. Here, we shall characterize
the box splines associated to a non-integer matrix satisfying the
refinement equation with some non-integer dilation and non-integer
translations.

\begin{theorem}\label{th:bsp}
Suppose that the univariate box spline $B(x|A)$ satisfies the
refinement  equation
$$
B(x|A)=\sum_{j=0}^N c_jB(\lambda x-d_j|A),
$$
 where $\lambda >1,\lambda\notin {\bf Z}$ and $d_j\in {\bf R}$.
 Then

 (A) There exists a positive integer $l$ such that $\lambda^l\in {\bf Z}.$

(B) For each element  $m_0\in A$, there exists an element  $m\in A
$ and an integer $p$ such that $m=p m_0/\lambda$.

 (C) The vector $A$ contains a sub-vector  of the
 form
 $$(m_0,m_0p_1 \lambda^{-1},\ldots,m_0p_{l-1}\lambda^{-l+1}),$$
 where $l$ is  a positive integer  such that $\lambda^l\in {\bf Z}$,
 and $p_j, j=1,\ldots,l-1,$ are integers satisfying $p_j|p_{j+1}$,
 $j=1,2,\dots,l-2,$ and $p_{l-1}|\lambda^l$.

(D) If all the $d_j\in {\bf Z}$, then the vector $A$ can be
represented as the union of the vectors of the form  $(m_0,\lambda
m_0 ,\ldots,\lambda^{k-1} m_0)$,
 where $k$ is the smallest  positive integer for which $\lambda^k\in {\bf
 Z}$ and $m_0\in {\bf Z}\setminus 0$.

\end{theorem}

Moreover,  using the technique of dimensional reduction, we can
extend the above results to higher dimensions.

\begin{theorem}\label{th:mulrespline}
Suppose that the  $s$-variable box spline $B(x|M)$ satisfies the
refinement  equation
$$
B(x|M)=\sum_{j=0}^N c_jB(\lambda x-d_j|M),
$$
 where $\lambda >1,\lambda\notin {\bf Z}$ and $d_j\in {\bf R}^s$.
 Then

 (A) There exists a positive integer $l$ such that $\lambda^l\in {\bf Z}.$

(B) For each column  $m_0$ in $M$, there exists a  column  $m$ in
$M$ and an integer $p$ such that $m=p m_0/\lambda$.

 (C) The matrix $M$ contains a sub-matrix  of the
 form
 $$(m_0,m_0p_1 \lambda^{-1},\ldots,m_0p_{l-1}\lambda^{-l+1})\in {\bf R}^{s\times l},$$
 where $l$ is  a positive integer  such that $\lambda^l\in {\bf Z}$,
 and $p_j, j=1,\ldots,l-1,$ are integers satisfying $p_j|p_{j+1}$,
 $j=1,2,\dots,l-2,$ and $p_{l-1}|\lambda^l$.

 (D) If all the $d_j\in {\bf Z}^s$, then the matrix $M$ can be
 represented as a disjoint union
 of matrices of the form
 $(m_0,\lambda m_0 ,\ldots, \lambda^{k-1} m_0)
 \in {\bf R}^{s\times k}$,
 where $k$ is the smallest  positive integer  such that
 $\lambda^k\in {\bf Z}$ and $m_0\in {\bf Z}^s\setminus 0$.

\end{theorem}

We conclude this section by stating a conjecture which classifies
all refinable splines in terms of box splines. A function $P(w)$
is {\it a real quasi-trigonometric polynomial} if it is of the
form of $\sum_{j=0}^N c_j e^{-2\pi i d_j w},$ where $c_j, d_j\in
{\bf R}$. The real quasi-trigonometric polynomial $P(w)$ is {\it
$\lambda$-closed} if $P(\lambda w)/P(w)$ is also a real
quasi-trigonometric polynomial.

\begin{conjecture}
The spline $f(x)$ satisfies  the refinement  equation
$$
f(x)=\sum_{j=0}^N c_j f(\lambda x-d_j),\,\,\,  \sum_{j=0
}^N c_j=\lambda,
$$
if and only if $\widehat{f}(w)$, i.e., the Fourier transform of
$f(x)$, can be expressed in terms of box splines as follows:
$$
\widehat{f}(w)=e^{2\alpha \pi i w}P(w)\widehat{B}(w|A),
$$
where $\alpha$ is some constant, $\widehat{B}(w|A)$ is the Fourier
transform of a $\lambda$-refinable box spline $B(x|A)$, and $P(w)$
is a $\lambda$-closed quasi-trigonometric polynomial with
$P(0)=1$.
\end{conjecture}

\begin{remark}
Throughout the paper, we only consider the case where $\lambda>0$.
However, as pointed out in \cite{daispline}, if $f(x)$ is a
$\lambda$-refinable spline then it is also a $(-\lambda)$-refinable
spline. Hence, the case, where $\lambda<0,$ can be also studied
using the results of this paper.
\end{remark}

\section{Proof of Theorem \ref{th:noninf}}

Before giving the proof of the  theorem, we need a theorem of purely
number-theoretic nature which plays an important role in the proofs
below. For any $x\in {\bf R},$ let $\|x\|_{\bf Z}$ denote the
distance from $x$ to the nearest integer in ${\bf Z}$.

\begin{theorem}\label{th:number}
Suppose that $\lambda>1$ and $r_1,\ldots,r_m$ are real numbers.
Then there exists a positive number $\xi$ and a positive number
$c=c(\lambda,r_1,\ldots,r_m)$ such that
$$
\|\xi\lambda^j-r_i\|_{\bf Z} \geq c
$$
for every $i=1,\ldots,m$ and every $j=0,1,2,\ldots $.
\end{theorem}
\begin{proof}
  For each $\tau>0,$ put
$$S_\tau:=\cup_{k \in {\bf Z}} \cup_{i=1}^m (k+r_i-\tau, k+r_i+\tau).$$
Then $\|x-r_i\|_{\bf Z} \geq \tau$ for every $i=1,\ldots,m$ if and
only if $x \in {\bf R} \setminus S_{\tau}.$ This implies that
$\|\xi \lambda^n-r_i\|_{\bf Z} \geq c$ (for every $i=1,\ldots,m$
and every integer $n \geq 0$) if and only if $\xi \notin S_c
\lambda^{-n}$ for $n \geq 0.$

We will construct a sequence of closed nested intervals $I_0
\supseteq I_1 \supseteq I_2 \supseteq \dots$ (where $I_0 \subset
(0,1)$) of length $|I_M| = \sqrt {2\lambda c}/\lambda^{gM},$ where
$M=0,1,2,\ldots,$ such that $\zeta \lambda^n \notin S_c$ for every
$\zeta \in I_M$ and every $n=-1,0,\dots,gM-1.$ (Here, for
convenience, we start with $n=-1,$ so the final result holds for
every $i=1,\ldots,m$ and every $n=-1,0,1,\ldots.$) Also, here
$g=g(\lambda,m)$ is the least positive integer satisfying
$$\lambda^g \geq 2(1+gm).$$ Then, for the common point
$\xi \in \cap_{M=0}^{\infty} I_M,$ the inequalities $\|\xi
\lambda^n -r_i\|_{\bf Z} \geq c$ will be satisfied for each
$i=1,\dots,m$ and each integer $n \geq -1.$

The proof is by induction on $M$. We begin with $M=0$. Evidently,
$S_c \cap (0,1)$ is a union of at most $m+1$ intervals of total
length $2cm.$ Thus $S_c \lambda \cap (0,1)$ is a union of $\leq
m+1$ intervals of total length $2\lambda cm.$ The set $(0,1)
\setminus S_c \lambda$ is thus a union of at most $m+2$ intervals
whose lengths sum to a number $\geq 1-2\lambda cm.$ It contains a
closed interval of length $\sqrt{2\lambda c}$ if $1-2\lambda
cm>(m+2)\sqrt{2\lambda c}.$ So the required closed interval $I_0
\subset (0,1)$ of length $\sqrt{2\lambda c}$ exists if
$$2\lambda cm+(m+2)\sqrt{2\lambda c} <1.$$
This inequality clearly holds if $c$ is less than a certain
constant depending on $\lambda$ and $m$ only. Let us start with
this $I_0.$

For the induction step $M \mapsto M+1,$ we assume that there exist
closed intervals $I_M \subset I_{M-1} \subset \dots \subset I_0$
with required lengths such that $I_M \subset (0,1) \setminus S_c
\lambda^{-n}$ for each $n=-1,0,\ldots,gM-1.$ We need to show that
$I_M$ contains a subinterval $I_{M+1}$ of length $\sqrt {2\lambda
c}/\lambda^{g(M+1)}$ such that $I_{M+1} \subset (0,1) \setminus S_
c \lambda^{-n}$ for each $n=gM,gM+1,\ldots,g(M+1)-1.$

Fix $n \in \{gM, gM+1,\ldots, g(M+1)-1\}.$ At most
$m(1+\lambda^n|I_M|)$ points of the form $(k+r_i)\lambda^{-n},$
where $k \in {\bf Z}$ and $i=1,\ldots,m,$ lie in $I_M.$ So the
intersection of $S_c\lambda^{-n}$ and $I_M$ consists of at most
$m(1+\lambda^n|I_M|)$ open intervals of length $2c/\lambda^n$ (or
less) each plus at most two intervals of length $c/\lambda^n$ (or
less) each at both ends of $I_M.$ As $n$ runs through
$gM,\ldots,g(M+1)-1$ the total length of such intervals is at most
$$\ell_M=\sum_{n=gM}^{g(M+1)-1} (m(1+\lambda^n|I_M|)2c\lambda^{-n}+2c\lambda^{-n})=
\sum_{n=gM}^{g(M+1)-1} (2cm|I_M|+2c(m+1)\lambda^{-n}).$$ Using
$\sum_{n=gM}^{g(M+1)-1} \lambda^{-n}<\sum_{n=gM}^{\infty}
\lambda^{-n}=\lambda^{1-gM}/(\lambda-1)=\lambda |I_
M|/(\sqrt{2\lambda c}(\lambda-1)),$ we find that
$$\ell_M< 2cmg|I_M|+2c(m+1)\lambda |I_M|/(\sqrt{2\lambda c}(\lambda-1))=
|I_M|(2cmg+\sqrt{2\lambda c}(m+1)/(\lambda-1)).$$

The remaining part in $I_M$ is of length at least
$$|I_M|- |I_M|(2cmg+\sqrt{2\lambda c}(m+1)/(\lambda-1))=
|I_M|(1-2cmg-\sqrt{2\lambda c}(m+1)/(\lambda-1)).$$ It consists of
at most $$1+\sum_{n=gM}^{g(M+1)-1} m(1+\lambda^n |I_M|)<1+gm+|I_
M|\lambda^{g(M+1)}/(\lambda-1)=1+gm+\sqrt{2\lambda c}
\lambda^g/(\lambda-1)$$ closed (possibly degenerate $[u,u]$)
intervals. In order to show that one of these closed intervals is
of length at least  $\sqrt{2\lambda c}/\lambda^{g(M+1)}$ (so that
we can take it as $I_{M+1}$) we need to check that
$$|I_M|(1-2cmg-\sqrt{2\lambda c}(m+1)/(\lambda-1)) \geq
|I_{M+1}|(1+gm+\sqrt{2\lambda c} \lambda^g/(\lambda-1))$$ for $c$
small enough. Indeed, using $|I_{M+1}|/|I_M|=\lambda^{-g},$ we can
rewrite this inequality as $$2cmg+(1+gm)\lambda^{-g}+
\sqrt{2\lambda c}(m+2)/(\lambda-1)<1.$$ Since $\lambda^g \geq
2(1+gm),$ the required inequality would follow from
$$2cmg+\sqrt{2\lambda c}(m+2)/(\lambda-1)<1/2.$$
Clearly, $g$ depends on $\lambda$ and $m$ only. So this inequality
holds for some positive $c$ depending on $\lambda$ and $m$ only.
This completes the proof of the theorem.
\end{proof}

\begin{remark}
The $m=1$ case of this theorem with some explicit constant $c$ was
recently obtained by the first named author in \cite{scan}. The
existence of such a positive number $c:=c(\lambda)$ for $m=1$ was
conjectured by Erd\"os \cite{erd} and then proved independently by
de Mathan \cite{dem} and Pollington \cite{pol1}. In fact, this
result was proved already by Khintchine in 1926 (see Hilfssatz III
in \cite{khi}), but then forgotten.
\end{remark}
\begin{remark}
 By estimating $g$ from above and by some standard
calculations, one can see that the inequalities $2\lambda
cm+(m+2)\sqrt{2\lambda c} <1$ (corresponding to the $M=0$ case)
and $2cmg+\sqrt{2\lambda c}(m+2)/(\lambda-1)<1/2$ (corresponding
to the induction step $M \mapsto M+1$ case) both hold if, for
instance, $c:=(\lambda-1)^2/(20(m+2)^2\lambda^3).$ This gives an
explicit expression for $c$ in Theorem~\ref{th:number} for each
$\lambda
>1$. The main part for the ``difficult" case when $\lambda$ is
close to $1$ is the factor $(\lambda-1)^2.$ It is essentially the
same factor as that in the proof of a similar result obtained by
the first named author in the $m=1$ case \cite{scan}.
\end{remark}

We now can give the proof of Theorem \ref{th:noninf}.

\begin{proof}[Proof of Theorem \ref{th:noninf}]
Suppose that $f\in {\bf C}^{\infty} $ is the solution of
$$
f(x)=\sum_{j=0}^N c_j f(\lambda x-d_j)
$$
with compact support. Take the Fourier transform of its both
sides. We obtain
\begin{equation}\label{eq:fi}
\widehat{f}(\xi)=H(\xi/\lambda)\widehat{f}(\xi/\lambda),
\end{equation}
where $H(\xi)={\lambda}^{-1} \sum_{j=0}^N c_j e^{-2\pi i d_j\xi}$.
By Paley and Wiener theorem, $\widehat{f}(\xi)$ is an entire
function satisfying $\widehat{f}(\xi)\leq C_k|\xi|^{-k}$ for any
positive integer $k$ and $\xi\in {\bf R},$ where $C_k$ is a
constant.

By (\ref{eq:fi}),  for any $M\in {\bf N},$ taking the product
$\widehat{f}(\lambda \xi)\widehat{f}(\lambda^2 \xi) \cdots
\widehat{f}(\lambda^M \xi),$ we deduce that
\begin{equation}\label{eq:fd}
|\widehat{f}(\lambda^M \xi)|=|\widehat{f}(\xi)|\prod_{j=0}^{M-1}
|H(\lambda^j\xi)|.
\end{equation}
Suppose that the zero points of $H(\xi)$ on $[0,D_0]$ are
$\{r_1,\ldots, r_m\}$, where $D_0$ is the least common multiple of
the denominators of $d_j$. Hence, each nonnegative root of $H(\xi)$
has the form of $r_i+kD_0$ with some $k\in {\bf Z}$. By Theorem
\ref{th:number}, we can select a positive number $\xi_0$ such that
there exists a $c>0$ for which
$$
\|\xi_0\lambda^j-r_i-kD_0\|_{\bf Z} = \|\xi_0\lambda^j-r_i\|_{\bf
Z} \geq c
$$
for every $i=1,\ldots,m,$  and every $j=0,1,2,\ldots $. Therefore,
there exists an $\varepsilon_0>0$ such that
$H(\lambda^j\xi_0)>\varepsilon_0$ for all $j$. So, by
(\ref{eq:fd}),
$$
|\widehat{f}(\lambda^M \xi_0)|\geq |\widehat{f}(\xi_0)|
\varepsilon_0^M.
$$
Moreover, since $f\in {\bf C}^\infty$, for any integer $k$, we can
find a constant $C_k$ such that
$$
C_k {(\lambda^M\xi_0)^{-k}}\geq |\widehat{f}(\lambda^M \xi_0)|\geq
|\widehat{f}(\xi_0)|\varepsilon_0^M.
$$
It follows that
$$
|\widehat{f}(\xi_0)|\leq C_k
{(\lambda^{-k})^M\xi_0^{-k}}{\varepsilon_0^{-M}}.
$$
We can select $k$ so large that $\lambda^{-k}< \varepsilon_0$.
Letting $M\rightarrow \infty$, we deduce that
$\widehat{f}(\xi_0)=0$.

Consider the derivative of $\widehat{f}(\xi)$ at $\xi_0$. Note
that
$$
\lambda
\widehat{f}'(\lambda\xi)=H(\xi)\widehat{f}'(\xi)+H'(\xi)\widehat{f}(\xi).
$$
Since $\widehat{f}(\xi_0)=0$, we have
$$
\lambda \widehat{f}'(\lambda\xi_0)=H(\xi_0)\widehat{f}'(\xi_0).
$$
It follows that
$$
|\lambda^M \widehat{f}'(\lambda^M
\xi_0)|=|\widehat{f}'(\xi_0)|\prod_{j=0}^{M-1}
|H(\lambda^j\xi_0)|.
$$
Using  essentially the same argument, we deduce that
$\widehat{f}'(\xi_0)=0$. By induction, it follows that
$\frac{d^k\widehat{f}(\xi)}{d\xi^k}|_{\xi=\xi_0}=0$ for any
integer $k \geq 0$. Since $\widehat{f}(\xi)$ is an entire
function, this yields $\widehat{f}(\xi)\equiv 0$ and hence
$f(x)\equiv 0$. This completes the proof.
\end{proof}

\section{Proofs of Theorem  \ref{th:gene}  and Proposition \ref{pr:exam}}

To prove Theorem \ref{th:gene}, we first prove a lemma, which
shows that the regularity of a refinable function can be
determined by a certain smoothness property near the endpoints of
the support.

\begin{lemma}\label{th:regul}
Let $f$ be a compactly supported solution of
\begin{equation}
f(x)=\sum_{j=0}^Nc_j f(\lambda x-d_j),
\end{equation}
with $supp \, f=[A,B]$. Suppose that there exist a $u_0\in [A,B]$
and $m \in {\bf N}$ such that $f^{(m)}(x)$ either does not exist
or is discontinuous at $u_0$. Then for any $\varepsilon>0$ there
are $x_0\in [A,A+\varepsilon)$ and $x_1\in (B-\varepsilon,B]$ such
that $f^{(m)}(x)$ either does not exist or is discontinuous at
$x_0$ and $x_1$.
\end{lemma}

\begin{proof}
 We first consider the case, where $x_0\in
[A,A+\varepsilon)$.
 Without loss of generality, we may suppose that $d_0=0$ (see Remark 1). By the
result of \cite{D1} (which was stated in Section 1), we have ${\rm
supp} f\subset [0,d_N(\lambda-1)^{-1}]$. To prove the lemma, note
that
$$
f(\lambda^{-1}x)=\sum_{j=0}^N c_jf(x-d_j).
$$
Hence
\begin{equation}\label{eq:revistwo}
f(x)={c_0}^{-1}f(\lambda^{-1}x)-{c_0}^{-1}\sum_{j=1}^Nc_jf(x-d_j).
\end{equation}
Consider the set
$$
S:=\{y|\,\, f^{(m)}(x) \mbox{ either does not exist or is
discontinuous at } y \}.
$$
 Since $f\notin {\bf C}^{m}$, the set $S$ is nonempty. We choose
 an element $y_1\in
 S$. By (\ref{eq:revistwo}),
 $\{{y_1}/{\lambda}, y_1-d_1,\ldots,y_1-d_N\}\cap S$ is
 nonempty. Therefore, we can find a $y_2\in \{{y_1}/{\lambda}, y_1-d_1,\ldots,y_1-d_N\}\cap S$ such that
 $0\leq y_2\leq y_1$. By induction, there is a sequence $y_k$ such that $y_k\in
 S$ and $0\leq y_k\leq y_{k-1}$. If this sequence contains $0$, then the
 lemma is proved. Suppose that $0$ is not an element of the
 sequence. By the construction of $y_k$, there exists  a $k_0$ such that
 $y_{k_0}<d_1$.
We have ${y_{k_0}}/{\lambda}\in S$, since $y_{k_0}-d_j<0$ for $1\leq
j \leq N$. Thus, for $k>k_0$, we can take $y_k=y_{k-1}/\lambda$. It
is clear that, for any $\varepsilon>0,$ we can find an integer  $k$
so large that $x_0:=y_{k_0}/\lambda^k<\varepsilon$.

In case $x_1\in (B-\varepsilon,B]$, we can suppose that $d_N=0$. By
the same method as above, the conclusion follows.
\end{proof}

Next, we give the proof of  Theorem \ref{th:gene}.

\begin{proof}[Proof of Theorem \ref{th:gene}]
We suppose that $f(x)$ is smooth on $(a_j,a_{j+1})$ where $1\leq
j\leq M$ and ${\rm supp} f\subset [a_1,a_{M+1}]$. Without loss of
generality, we can suppose that $d_0=0$, and hence $a_1=0$. Let us
define
$$
(\frac{d}{dx})^k f_{-}(a_j)=\lim_{x\rightarrow
a_j-}(\frac{d}{dx})^k f(x),\,\,\,\, (\frac{d}{dx})^k
f_{+}(a_j)=\lim_{x\rightarrow a_j+}(\frac{d}{dx})^k f(x),
$$
and $f_k(a_j)=(\frac{d}{dx})^k f_+(a_j)-(\frac{d}{dx})^k f_-(a_j)$.
We shall prove that either $f_k(a_j)=0$ or $\lim_{x\rightarrow
a_j}(\frac{d}{dx})^kf(x)$ exists for all $a_j$ except when $k=k_0$
for some nonnegative integer $k_0$. Note that ${\rm supp} f\subset
[0,a_{M+1}]$. Then the function $f(x)$ satisfies
\begin{equation}\label{eq:prth1}
f(x)=c_0f(\lambda x-d_0)=c_0f(\lambda x),\,\,\,\, x\in
[0,\varepsilon]
\end{equation}
 for a sufficiently small $\varepsilon>0$.

We claim that there exists a nonnegative integer $k_0$ such that
$f_+^{(k_0)}(0)\neq 0$. For a contradiction, assume that
$f_+^{(k)}(0)=0$ for any nonnegative integer $k$. Note that
$f_-^{(k)}(0)=0$ for each $k \geq 0,$ since $f(x)=0$ for $x<0$. So
$f^{(k)}(x)|_{x=0}$ exists for any nonnegative integer $k$. Since
$f(x)$ is the piecewise smooth function, there exists a positive
number $\varepsilon_1$ such that $f(x)\in {\bf C}^{\infty}
[0,\varepsilon_1)$. According to Lemma \ref{th:regul}, we have
$f(x)\in {\bf C}^{\infty}$. But then Theorem \ref{th:noninf}
implies that $f(x)\equiv 0$, a contradiction.

According to (\ref{eq:prth1}), we have
$f_+^{(k_0)}(0)=c_0\lambda^{k_0}f^{(k_0)}_+(0)$. But
$f^{(k_0)}_+(0)\neq 0,$ so $c_0=\lambda^{-k_0}$. By
(\ref{eq:prth1}), we  conclude that $f(x)=\lambda^{-k_0}f(\lambda
x)$ on $[0,\varepsilon]$. A simple calculation shows that
$$
f^{(k)}_+(0)=\lambda^{k-k_0}f^{(k)}_+(0) \mbox{ for any } k\in
{\bf Z}_+.
$$
Clearly, $\lambda\neq 1$ yields that $f^{(k)}_+(0)=0$ for $k\neq
k_0$. Note that $f^{(k)}_{-}(0)=0$ for any nonnegative integer $k$.
Accordingly, $\lim_{x\rightarrow 0}f^{(k)}(x)$ exists for $k\neq
k_0$. By Lemma \ref{th:regul}, $\lim_{x\rightarrow a_j}f^{(k)}(x)$
exists for all $a_j$ if $k\neq k_0$. Put $g(a_j)=\lim_{x\rightarrow
a_j}f^{(k_0+1)}(x)$ and $g(x)=f^{(k_0+1)}(x)$ for $x\neq a_j$. We
will show that $g(x)$ satisfies the following refinement equation
$$
g(x)=\lambda^{k_0+1}\sum_{j=0}^Nc_jg(\lambda x-d_j).
$$

Fix $x_0\in {\bf R}$. If $\{x_0,\lambda x_0-d_j,j=1, \ldots,
N\}\cap \{a_j, j=1, \ldots, M+1\}=\emptyset$, by taking the
$(k_0+1)$-th derivative at $x_0$ on both sides of
$$
f(x)=\sum_{j=0}^N c_j f(\lambda x-d_j),
$$
we obtain
\begin{equation}
g(x_0)=\lambda^{k_0+1}\sum_{j=0}^Nc_jg(\lambda x_0-d_j).
\end{equation}

Let us consider the remaining case when the intersection of two
sets is non-empty. Without loss of generality, we may suppose that
$$
\{x_0,\lambda x_0-d_j,j=1, \ldots, N\}\cap \{a_j, j=1, \ldots,
M+1\}=\{x_0,\lambda x_0-d_j,j=1,\ldots,N_0\}$$ with an integer
$N_0$.
 Select an
$\varepsilon>0$ such that $$(x_0-\varepsilon,x_0)\cap \{a_j, j=1,
\ldots, M+1\}=\emptyset $$ and $(\lambda x_0-d_k-\lambda
\varepsilon,\lambda x_0-d_k)\cap \{a_j, j=1, \ldots,
M+1\}=\emptyset, \mbox{ for }  k=1, \ldots, N$.

Hence $f(x),f(\lambda x-d_j)\in {\bf C}^\infty
(x_0-\varepsilon,x_0)$, where  $j=1,\ldots,N$. Then, for $x\in
(x_0-\varepsilon,x_0)$, we have
\begin{equation}\label{eq:x0eq1}
g(x)-\lambda^{k_0+1}\sum_{j=0}^{N_0} c_jg(\lambda
x-d_j)=\lambda^{k_0+1}\sum_{j=N_0+1}^Nc_jg(\lambda x-d_j).
\end{equation}
By taking the limits on both sides of (\ref{eq:x0eq1}), and noting
that
$$
\lim_{x\rightarrow x_0-}\lambda^{k_0+1}\sum_{j=N_0+1}^Nc_jg(\lambda
x-d_j)=\lambda^{k_0+1}\sum_{j=N_0+1}^Nc_jg(\lambda x_0-d_j)
$$
and
$$
 \lim_{x\rightarrow x_0-}(g(x)-\lambda^{k_0+1}\sum_{j=0}^{N_0}
c_jg(\lambda x-d_j))=g(x_0)-\lambda^{k_0+1}\sum_{j=0}^{N_0}
c_jg(\lambda x_0-d_j),
$$
we conclude that
$$
g(x_0)=\lambda^{k_0+1}\sum_{j=0}^Nc_jg(\lambda x_0-d_j).
$$
Combining the results above, we arrive to the equality
$$
g(x)=\lambda^{k_0+1}\sum_{j=0}^Nc_jg(\lambda x-d_j), {\,\, for\,\,
all\,\, } x\in {\bf R}.
$$

Next, by Theorem \ref{th:noninf}, we obtain that $g(x)\equiv 0$,
since the function $g\in {\bf C}^{\infty}$ is compactly supported.
Using $f^{(k_0+1)}|_{(a_j,a_{j+1})}(x)=g|_{(a_j,a_{j+1})}(x)\equiv
0$, we conclude that $f(x)$ is a spline function.
\end{proof}

Now, we begin the proof of Proposition \ref{pr:exam}. For this, we
shall give some definitions (see also \cite{daispline}) and a
lemma. The functions of the form $G(w)=\sum_{j=0}^N a_je^{-2\pi i
b_jw}$ are referred to as {\it quasi-trigonometric polynomials},
where $a_j\in {\bf C},a_j\neq 0$ and $b_j\in {\bf R}$ such that
$b_0<b_1<\cdots <b_N$. In case $b_j \in {\bf Z},$ such polynomials
are simply trigonometric polynomials. If $b_0=0$, $G(w)$ is called
a {\it normalized quasi-trigonometric polynomial}. For the
quasi-trigonometric polynomial $G(w)$, one can write
\begin{equation}\label{eq:deom}
G(w)=e^{-2\pi i r_1 w}G_1(w)+e^{-2\pi i r_2 w}G_2(w)+\cdots
+e^{-2\pi i r_l w}G_l(w),
\end{equation}
where each $G_j$ is a trigonometric polynomial and $0\leq
r_1<\cdots < r_l<1$ are distinct. It is easy to see that up to a
permutation of terms this decomposition is unique and we shall
call (\ref{eq:deom}) the {\it standard decomposition} of $G(w)$.
Moreover, as in \cite{daispline}, we write ${\bf
A}_G(w):=\sum_{j=0}^H c_j e^{-2\pi i k_j w}$ for the greatest
common divisor of the trigonometric polynomials $\{G_j\}$
normalized so that $k_0=0,c_0=1$ and $k_j\geq 0$ are distinct. It
is easy to see that ${\bf A}_G(w)=\alpha e^{-2\pi i j w} G(w)$ for
some constant $\alpha$ and an integer $j$, if and only if, $G(w)$
is a trigonometric polynomial. Then we have

\begin{lemma}\label{le:factor}
Suppose that two normalized quasi-trigonometric polynomials $G_1$
and $G_2$ satisfy the equality $G_1(w)G_2(w)=R(w)(1-e^{-2\pi
iw})$, where $R(w)=1+\sum_{j=1}^Na_je^{-2\pi i b_j w}$ and
$0<b_1<b_2<\cdots<b_n$ are  irrational numbers. Then there exists
a positive integer $P$ and a set $S_0\subset S_1:=\{0,1,\ldots,
P-1\}$ such that $G_1(w)=\alpha(w) \prod_{j\in S_0} (e^{-2\pi i
j/P}-e^{-2\pi i w/P})$ and $G_2(w)=\beta(w)\prod_{j\in
S_1\setminus S_0} (e^{-2\pi i j/P}-e^{-2\pi i w/P})$, where
$\alpha$, $\beta$ are some quasi-trigonometric polynomials.
\end{lemma}

\begin{proof} Set $G_{P}(w):=R(Pw)(1-e^{-2\pi iPw})$,  $G_{1P}(w):=G_1(Pw)$ and $G_{2P}(w):=G_2(Pw)$
for any integer $P$. Note that ${\bf A}_{G_{P}}(w)=1-e^{-2\pi
iPw}$.
 We claim that the conclusion holds if there exists a positive integer
$P$ such that
\begin{equation}\label{eq:proof1}
 {\bf A}_{G_{1P}}(w){\bf A}_{G_{2P}}(w)=1-e^{-2\pi i Pw}.
\end{equation}

Indeed, on both sides of (\ref{eq:proof1}) we have trigonometric
polynomials. Setting $z=e^{-2\pi iw}$, we see that $1-z^P$ is a
product of two polynomials in $z$. So  there exists a set
$S_0\subset \{0,1,\ldots,P-1\} $ such that ${\bf
A}_{G_{1P}}(w)=\prod_{j\in S_0}(e^{-2\pi i j}-e^{-2\pi i w})$,
which implies $G_1(w)=\prod_{j\in S_0}(e^{-2\pi i j/P}-e^{-2\pi i
w/P})\alpha(w)$, where $\alpha(w)$ is a quasi-trigonometric
polynomial. Similarly, we have $G_2(w)=\prod_{j\in S_1\setminus
S_0}(e^{-2\pi i j/P}-e^{-2\pi i w/P})\beta(w)$, where $\beta(w)$
is a quasi-trigonometric polynomial. The claim follows
immediately. Hence, in order to complete the proof, it suffices to
show that there exists a positive integer $P$ such that ${\bf
A}_{G_{1P}}(w){\bf A}_{G_{2P}}(w)=1-e^{-2\pi i Pw}$.

Write $G_1(w)=\sum_{j=0}^{N_1} a_{1j}e^{-2\pi i b_{1j}w}$ and
$G_2(w)=\sum_{j=0}^{N_2} a_{2j}e^{-2\pi i b_{2j}w}$. Then there
exists a positive integer, say, $P$ such that each element in the
set
$$\{Pb_{1j}, Pb_{2k}, P(b_{1j}-b_{1j'}),
P(b_{2k}-b_{2k'}) \>|\> 0 \leq j, j' \leq N_1, \> 0 \leq k, k'
\leq N_2\}$$ is either integer or irrational. Write
\begin{equation}\label{eq:fact}
R(w)(1-e^{-2\pi  i Pw})=G_{1P}(w)G_{2P}(w).
\end{equation}
We claim that for each
 $w_0$ satisfying $1-e^{-2\pi iPw_0}=0$, one has
 $${\bf A}_{G_{1P}}(w_0){\bf
 A}_{G_{2P}}(w_0)=0.$$
 Suppose
 that there is an element $w_0\in \{w|1-e^{-2\pi i Pw}=0\}$
 such that neither ${\bf A}_{G_{1P}}(w_0)$ nor ${\bf
 A}_{G_{2P}}(w_0)$ is zero. Write $w_0$ in the form
 $j_0/P+I_0$, where $0\leq j_0<P$,  $j_0,I_0\in {\bf Z}$.
 Then ${\bf A}_{G_{1P}}(j_0/P+I) \ne 0$ and ${\bf
 A}_{G_{2P}}(j_0/P+I) \ne 0$ for every
$I\in {\bf Z}$.

On the other hand, for each $I\in {\bf Z}$, one has
$G_{1P}(j_0/P+I)G_{2P}(j_0/P+I)=0$ (see (\ref{eq:fact})). It
follows that there exists an infinite set ${\bf Z}_0\subset {\bf
Z}$, such that $G_{1P}(j_0/P+I)=0$  for each $I\in {\bf Z}_0$
(otherwise, one can replace $G_{1P}$ by $G_{2P}$). According to
the choice of $P$, we have the following decomposition
$$
G_{1P}(w)=G_1(Pw)=e^{-2\pi i r_1 w}Q_1(w)+e^{-2\pi i r_2
w}Q_2(w)+\cdots +e^{-2\pi i r_l w}Q_l(w),
$$
where $0\leq r_1<\ldots <r_l$ are all irrational numbers (except
perhaps for $r_1=0$) and $Q_j(w)$ are trigonometric polynomials.
Moreover, by the choice of $P$ all the differences $r_i-r_j$ are
also irrational. Substituting $j_0/P+I$ with $I\in {\bf Z}_0$ into
$G_{1P}(w)$ we have
\begin{equation}\label{eq:vand}
G_{1P}(j_0/P+I)=\sum_{k=1}^l e^{-2\pi i r_k I} e^{-2\pi i
r_kj_0/P}Q_k(j_0/P)=\sum_{k=1}^l e^{-2\pi i r_k I} V_k=0,
\end{equation}
where $V_k=e^{-2\pi i r_kj_0/P}Q_k(j_0/P)$ are all independent of
$I$. Suppose that $I_1,\ldots, I_l$ are $l$ distinct elements of
${\bf Z}_0 $. (As ${\bf Z}_0 $ is infinite, one can find in it $l$
distinct elements.) Let $M$ be the $l\times l$ generalized
Vandermonde matrix
$$
M=(e^{-2\pi ir_k I_j})_{1\leq k,j\leq l}.
$$
Note that all the $r_k,$ where $k \geq 2,$ are irrational and all
the differences $r_k-r_j$ are irrational, so the matrix $M$ is
non-singular. Taking $I=I_1,\ldots, I_l$ in (\ref{eq:vand}), we
obtain $M{\bf V}=0$, where ${\bf V}:=(V_1,\ldots,V_l)^T$. It
follows that all $V_j=0$ and thus $Q_k(j_0/P)=0$ for every $k$.
Hence ${\bf A}_{G_{1P}}(j_0/P+I)=0$ for each $I\in {\bf Z}$, which
contradicts to ${\bf A}_{G_{1P}}(w_0){\bf A}_{G_{2P}}(w_0)\neq 0$
and proves the claim.

Moreover, for each $w_0$ such that ${\bf A}_{G_{1P}}(w_0){\bf
 A}_{G_{2P}}(w_0)=0$, one has $1-e^{-2\pi i P w_0}=0$. Indeed, if
$1-e^{-2\pi i P w_0}\neq 0$, then $1-e^{-2\pi i P (w_0+I)}\neq 0$
for any $I\in {\bf Z}$. So, $R(w_0+I)=0$ for each $I\in {\bf Z}$.
By a similar method, we can show that $R(w)\equiv 0$, which
contradicts to the definition of $R(w)$. Similarly, one can show
that the roots of ${\bf A}_{G_{1P}}(w){\bf
 A}_{G_{2P}}(w)=0$ are of multiplicity $1$.   Hence we have ${\bf
A}_{G_{1P}}(w){\bf
 A}_{G_{2P}}(w)=1-e^{-2\pi i P w}$. The lemma follows.
 \end{proof}

\begin{proof}[Proof of Proposition \ref{pr:exam}]
 Suppose $B(x|(1,\sqrt{5/2}))=f_1(x)*f_2(x)$, where $f_1(x)$ and
$f_2(x)$ are splines. By  the Fourier transform, we have
\begin{equation}\label{eq:decomp}
\widehat{f_1}(w)\widehat{f_2}(w)=\frac{1-e^{-2\pi i w}}{2\pi i w}
\frac{1-e^{-2\pi\sqrt{{5}/{2}} i w}}{2\pi i\sqrt{{5}/{2}} w}.
\end{equation}
By Corollary 2.2 of \cite{daispline}, $\widehat{f_j}(w)$ has the
form of $p_j(w)/w$, where $p_j(w)$ is a quasi-trigonometric
polynomial with $p_j(0)=0$ for $j=1,2$. Then we have
$$
p_1(w)p_2(w)=R(w)(1-e^{-2\pi i w}),
$$
where $R(w):=\frac{1-e^{-2\pi \sqrt{5/2}iw}}{-4\pi^2 \sqrt{5/2}}$.
 Lemma \ref{le:factor} with $S_0$ (or $S_1\setminus S_0$) containing $0$
 implies that there exists an integer $P_1$
such that either $p_1(w)$ (or $p_2(w)$) is of the form of
$(1-e^{-2\pi i w/P_1})q_1(w)$, where $q_1(w)$ is a
quasi-trigonometrical polynomial. Without loss of generality we
may assume that $p_1(w)=(1-e^{-2\pi i w/P_1})q_1(w)$. Note that
$w=0$ is a root of $p_1(w)$ and $p_2(w)$ of multiplicity $1$.
Hence, by a similar method as above applied to $w'=\sqrt{5/2} w$,
we obtain that $p_2(w)$ is of the form
 $p_2(w)=(1-e^{-2\pi i\sqrt{5/2} w/P_2})q_2(w)$, where $P_2$ is
a positive integer and $q_2$ is a quasi-trigonometrical
polynomial.

Set $Z(f):=\{w|\,\,f(w)=0, w\in {\bf C}\}$ and ${Z}'(f):=\{w|\,\,
f(\sqrt{10} w)=0, w\in {\bf C}\}$. We claim $Z(p_1)\setminus
{Z}'(p_1)\neq \emptyset$, i.e., there exists $w_0\in {\bf C}$ such
that $p_1(w_0)=0$ but $p_1(\sqrt{10} w_0)\neq 0$. To prove this,
note that
$$Z(p_1/q_1)=\{I_1P_1|\,\, I_1\in
{\bf Z}\}$$
 and $${Z}'(p_1q_2)=\{I_2/\sqrt{10},\,\, I_3P_2/5+k/5|\,\, I_2,
I_3 \in {\bf Z}, 1\leq k\leq P_2-1 \}.$$
 We can see that
$Z(p_1/q_1)\setminus {Z}'(p_1q_2)\neq \emptyset$. Since
$Z(p_1/q_1)\subset Z(p_1)$ and ${Z}'(p_1)\subset {Z}'(p_1q_2)$, we
derive that $Z(p_1)\setminus {Z}'(p_1)\neq \emptyset$, where
$(Z(p_1/q_1)\setminus {Z}'(p_1q_2))\subset (Z(p_1)\setminus
{Z}'(p_1))$.

Hence there exists a $w_0$ such that $p_1(w_0)=0$ while
$p_1(\sqrt{10} w_0)\neq 0$. As a result,  we see that
$p_1(\sqrt{10} w)/p_1(w)$ cannot be a mask polynomial. Thus, $f_1$
is not $\sqrt{10}$-refinable. The proposition follows.
\end{proof}

\section{Proofs of Theorem \ref{th:bsp} and Theorem \ref{th:mulrespline}}

\begin{proof}[Proof of Theorem \ref{th:bsp}]
Let us begin with part (B). Write $A=(m_1,\ldots,m_n)$, where
$m_j\in {\bf R}\setminus 0$.
 We shall prove that, for each element of $A$, for instance, $m_1$ there
exists an element  $m\in A $ and an integer $p$ such that $m=p
m_1/\lambda$.

Consider the Fourier transform
$$\widehat{B}(w|A)=\prod_{j=1}^n
\frac{1-e^{-2\pi i wm_j}}{2\pi i wm_j}.
$$
Since $B(x|A)$ satisfies the refinement  equation, we have
\begin{equation}\label{eq:four}
\widehat{B}(\lambda w|A)=p(w)\widehat{B}( w|A),
\end{equation}
where $p(w)$ is the mask polynomial. Note that
$$p(w)=\frac{\widehat{B}(\lambda
w|A)}{\widehat{B}(w|A)}=\lambda^{-n}\prod_{j=1}^n\frac{Q(\lambda
m_jw)}{Q( m_jw)},
$$
 where $Q(w)={1-e^{-2\pi i w}}$.
Put
$$
Z_{j}:=\{w|\,\, Q(m_jw)=0,w\neq 0\} \mbox{ and } Z_{j}':=\{w|\,\,
Q(\lambda m_j w)=0, w\neq 0\}.
$$
 Since $p(w)$ is an entire function, one has
\begin{equation}\label{eq:zero}
\bigcup_{j=1}^nZ_j \subset \bigcup_{j=1}^nZ_j'.
\end{equation}
A simple calculation shows that
$$Z_j:=\{{I_j}/{m_j}|\,\,  I_j\in {\bf
Z}\setminus 0\} \mbox{ and } Z_j':=\{{k_j}/(m_j\lambda)|\,\,
k_j\in {\bf Z}\setminus 0\}.$$
 Let us consider $Z_1$. Put
$J_1:=\{j|Z_1\cap Z_j'\neq \emptyset\}.$ By (\ref{eq:zero}), we see
that $Z_1\subset \bigcup_{j\in J_1}Z_j'$.  Select an entry in $J_1$,
say, $u$. Since $Z_1\cap Z_{u}'\neq \emptyset$, there exist $I_1\in
{\bf Z}\setminus 0$ and $k_{u}\in {\bf Z}\setminus 0$ such that
${I_1}/{m_1}=k_{u}/(m_{u}{\lambda})$. Hence, we can find two coprime
integers $P_{1u},Q_{1u}$ such that
$m_{u}={P_{1u}m_1}/(Q_{1u}{\lambda})$. Similarly, for any index
$u\in J_1$, we can find two coprime integers $P_{1u},Q_{1u}$ such
that
\begin{equation}\label{eq:mum1}
m_u\,\,=\,\, {P_{1u}m_1}/(Q_{1u}{\lambda}).
\end{equation}
 Note that for each fixed
$I_1\in {\bf Z}\setminus 0$ there exist $u\in J_1,k_u\in {\bf
Z}\setminus 0$ such that
$$
\frac{I_1}{m_1}\,\,=\,\,\frac{k_u}{m_{u}\lambda}\,\,=\,\,\frac{k_u
Q_{1u}}{m_1 P_{1u}},
$$
since $Z_1\subset \bigcup_{j\in J_1}Z_j'$.
 Hence, for any $I_1\in {\bf
Z}\setminus 0$, there exist  $u\in J_1$ and $k_u\in {\bf
Z}\setminus 0$ such that ${k_u}/{I_1}={P_{1u}}/{Q_{1u}},$ i.e.,
$k_u=I_1 {P_{1u}}/{Q_{1u}}$. So, there exists $u\in J_1$ such that
$Q_{1u}|I_1$ for any $I_1\in {\bf Z}\setminus 0$.

We claim that one can find a $u_1\in J_1$ such that
$Q_{1u_{1}}=1$. Indeed, suppose $Q_{1u}\neq 1$ for each $u\in
J_1$. Take $I_1=\prod_{u\in J_1}|Q_{1u}|+1$. But then $I_1$ is not
divisible by $Q_{1u}$ for each $u\in J_1$, a contradiction.

Next, by (\ref{eq:mum1}), there exists an index $u_1$ such that
$$
m_{u_1}=\frac{P_{1u_1}m_1}{\lambda }.
$$
Setting $m:=m_{u_1}$ and $p:=P_{1u_1}$ we complete the proof of
(B).

Consider parts (A) and (C). Using (B), one obtains an infinite
sequence $u_1,u_2,\ldots$, such that
$m_{u_k}={P_{u_ku_{k-1}}m_{u_{k-1}}}/{\lambda }$, for some
$P_{u_ku_{k-1}}\in {\bf Z}$, where $k\geq 2$.
  Since each index in
this sequence is at most $n$, the sequence contains two equal
indices. Without loss of generality, we may suppose that
$u_{l+1}=u_1$, where $l$ is a positive integer. Then
\begin{eqnarray}\label{eq:parta}
 m_{u_2}&=&\frac{1}{\lambda}P_{u_1u_2}m_{u_1},\nonumber\\
 m_{u_3}&=&\frac{1}{\lambda}P_{u_2u_3}m_{u_2},\nonumber\\
m_{u_4}&=&\frac{1}{\lambda}P_{u_3u_4}m_{u_3},\\
&\vdots& \nonumber\\
 m_{u_1}&=&\frac{1}{\lambda}P_{u_lu_1}m_{u_l}.\nonumber
 \end{eqnarray}

This  yields $\lambda^l=P_{u_1u_2}\cdots P_{u_lu_1} \in {\bf Z}$,
proving part (A). Moreover, by using (\ref{eq:parta}), one can see
that $m_{u_2}={m_{u_1} p_1}/{\lambda},m_{u_3}={m_{u_1}
p_2}/{\lambda^2},\ldots,m_{u_l}=p_{l-1}m_{u_1}/{\lambda^{l-1}}$,
where $p_1:=P_{u_1u_2}, p_{2}:=p_1
P_{u_2u_3},\ldots,p_{l-1}:=p_{l-2}P_{u_l,u_{l-1}}$. Hence
$p_j|p_{j+1}$ for
 $j=1,2,\dots,l-2,$ and $p_{l-1}|\lambda^l$. Thus the
sub-vector $(m_{u_1},\ldots,m_{u_l})$ has the form of $(v,vp_1
\lambda^{-1},\ldots,vp_{l-1}\lambda^{-l+1})$, where $v:=m_{u_1}$,
proving (C).

It remains to prove (D). Consider the Fourier transform of
$B(x|A)$
\begin{equation}\label{eq:box1}
\widehat{B}(w|A)=\prod_{j=1}^n\frac{1-e^{-2\pi i m_j w}}{2\pi i
m_j w}.
\end{equation}

Since $B(x|A)$ is a $\lambda$-refinable spline with integer
translations, by Theorem \ref{th:lawton} and Theorem
\ref{th:respline}, we can write $\widehat{B}(w|A)$ in the form
\begin{equation}\label{eq:box2}
\widehat{B}(w|A)=e^{-2\pi i z_0 w}\prod_{j=0}^{k-1} p(e^{2\pi
i\lambda^{j} w})(\frac{1-e^{-2\pi i \lambda^{j} w}}{2\pi i
\lambda^{j} w})^h,
\end{equation}
where $z_0:=d_0(1+\lambda+\dots+\lambda^{k-1})/(\lambda^k-1)$, $h$
is an integer and $p(z)$ is a polynomial.

Set
$$G_1(w):=\prod_{j=0}^{k-1} p(e^{2\pi i\lambda^{j}
w})({1-e^{-2\pi i \lambda^{j} w}})^h$$ and
$$G_2(w):=\prod_{j=1}^n
(1-e^{-2\pi i m_j w}).$$ Comparing (\ref{eq:box1}) and
(\ref{eq:box2}), one gets $G_2(w)=e^{2\pi i \alpha w}G_1(w)$, where
$\alpha$ is a constant. Expanding $\prod_{j=0}^{k-1} p(e^{2\pi
i\lambda^{j} w})({1-e^{-2\pi i \lambda^{j} w}})^h$, we see that each
term has the form $e^{2\pi iw \sum_{j=0}^{k-1} b_j\lambda^j}$ with
$b_j\in {\bf Z}$. Since $\lambda^k>1$ is an integer, by the theorem
of Capelli (see \cite{lang} or p.~92 in \cite{schi}), $\lambda$ is
an algebraic integer of degree $k$. Thus the difference between two
distinct numbers in the form of $\sum_{j=0}^{k-1} b_j\lambda^j$ is
non-integer. We conclude that ${\bf A}_{G_1}(w)=p(e^{2\pi i
w})(1-e^{-2\pi i w})^h$. A simple observation also shows that ${\bf
A}_{G_2}(w)=\prod_{m_j\in {\bf Z}} (1-e^{-2\pi i m_j w})$. Since
${\bf A}_{G_1}(w)={\bf A}_{G_2}(w)$, there are $h$ integer entries
in the vector $A$. Without loss of generality, suppose that
$m_1,\ldots,m_h\in {\bf Z}$. Note that
$$G_1(w)=\prod_{j=0}^{k-1} {\bf A}_{G_1}(\lambda^{j}w)=
\prod_{j=0}^{k-1} \prod_{r=1}^h (1-e^{-2\pi i m_r \lambda^{j}
w}).$$ Then $G_2(w)=e^{2\pi i \alpha w} \prod_{j=0}^{k-1}
\prod_{r=1}^h (1-e^{-2\pi i m_r \lambda^{j} w})$. It follows that
$\alpha =0$, so the vector $A$ can be written as a union of
$(m_r,\lambda m_r,\ldots, \lambda^{k-1} m_r),r=1,\ldots,h$,
proving part (D). \end{proof}

\begin{proof}[Proof of Theorem \ref{th:mulrespline}]
 We begin with (A). Assume that $$M=(m_1,\ldots, m_n)\in {\bf
R}^{s\times n}.$$ It is well-known that the Fourier transform of
$B(x|M)$ is
$$
\widehat{B}(w|M)=\prod_{j=1}^n \frac{1-e^{-2\pi i w^T m_j}}{2\pi i
w^T m_j}.
$$
Since $B(x|M)$ is $\lambda$-refinable, the mask polynomial $H$ is
given by the formulae $H(w)=\widehat{B}(\lambda
w|M)/\widehat{B}(w|M).$

Put $W:=\{w|w^T m_j\neq 0, \mbox{ for all } j \}$. Select a $w_0\in
W$ and consider
\begin{equation}\label{eq:fz}
\widehat{f}(z|w_0):=\widehat{B}(zw_0|M)=\prod_{j=1}^n
\frac{1-e^{-2\pi i zw_0^Tm_j}}{2\pi iz w_0^Tm_j},
\end{equation}
 where $z\in {\bf R}$. Observe
that $\widehat{f}(z|w_0)$ can be considered as the Fourier
transform of the univariate box spline
$B(x|(w_0^Tm_1,\ldots,w_0^Tm_n))$, which is $\lambda$-refinable
for each fixed $w_0\in W$ with the function ${\widehat{f}(\lambda
z|w_0)}/{\widehat{f}(z|w_0)}=H(zw_0)$ being a quasi-trigonometric
polynomial. Then, by Theorem \ref{th:bsp}, there exists a positive
integer $l$ such that $\lambda^l\in {\bf Z}$, proving part (A).

Select an entry in the vector $(w_0^Tm_1,\ldots,w_0^Tm_n)$, for
instance, $w_0^Tm_0$. Since $B(x|(w_0^Tm_1,$ $\ldots,w_0^Tm_n))$
is $\lambda$-refinable, one can find an integer $p_{w_0}$ and an
index $j_{w_0}$ satisfying $2\leq j_{w_0} \leq n$ such that
$w_0^Tm_{j_{w_0}}=p_{w_0}{w_0^Tm_0}/{\lambda}$ for any $w_0\in W$
(see part (B) in Theorem \ref{th:bsp}). Each $w_0$ corresponds to
an index $j_{w_0}$. For an index $j$, set $W(j)=\{w_0\in
W|j_{w_0}=j \}$. Then $\cup_{j=2}^n W(j)=W$. Since the $s$
dimensional Lebesgue outer measure of $W$ is infinite, there
exists a subset of $W$, say, $W_0$ such that, for any $w_0\in
W_0$, the index $j_{w_0}$ is a constant and the $s$ dimensional
Lebesgue outer measure  of $W_0$ is positive. We suppose the
constant  index $j_{w_0}$ is $u_1$. Then
\begin{equation}\label{eq:last}
p_{w_0}\frac{w_0^Tm_0}{\lambda}=w_0^Tm_{u_1}
\end{equation}
 for any $w_0\in W_0$.
We now consider all integers $p_{w_0}$. For each $q \in {\bf Z}$,
set
$$
W_0(q):=\{w_0|p_{w_0}=q,\ \ w_0\in W_0\}.
$$
Then
$$
\bigcup_{q\in {\bf Z}}W_0(q)=W_0.
$$
We claim that there exists a positive integer  $q_1$ such that the
$s$ dimensional  Lebesgue outer measure of $W_0(q_1)$ is positive.
Indeed, the $s$ dimensional Lebesgue outer measure  of $W_0$ is
positive and the set ${\bf Z}$ is countable. Hence there exist $s$
linearly independent elements of $W_0(q_1)$. Let us denote them by
$w_1,\ldots, w_s$. We have
$$
q_1\frac{w_j^Tm_0}{\lambda}=w_j^Tm_{u_1},
$$
for each $1\leq j\leq s$, by (\ref{eq:last}). It follows that
\begin{equation}\label{eq:inde}
w_j^T{\bf V}=0,
\end{equation}
 for each $1\leq j\leq s$ where ${\bf V}:=(q_1\frac{m_0}{\lambda}-m_{u_1})$.
 Let $A$ be the $s\times s$ matrix
 $$
A=(w_j^T)_{1\leq j\leq s}.
 $$
 Then $A$ is non-singular, because $w_j$ are linearly independent.
The equality (\ref{eq:inde}) can be written as $A{\bf V}=0$. It
follows that
 ${\bf V}=0$, because $A$ is non-singular.  Hence
$m_{u_1}={q_1m_0}/{\lambda}$. Putting $m:=m_{u_1},p:=q_1$, we
complete the proof of (B).

Part (C) can be proved by the same method as in the proof of the
part (C) of Theorem \ref{th:bsp}. We omit the details.

It remains to prove (D). Set $W':=W\cap {\bf Z}^s$. By the
definition of $W$, one can see that $W'={\bf Z}^s\setminus
(\cup_{j=1}^nH_j)$, where $H_j:=\{w\in {\bf Z}^s|w^Tm_j=0 \}.$
Then $B(x|(w_0^Tm_1,\ldots,w_0^Tm_n))$ is $\lambda$-refinable with
integer translations for every $w_0\in W'$. By Theorem
\ref{th:bsp}, the vector $(w_0^Tm_1,\ldots,w_0^Tm_n)$ can be
written as
$$
(w_0^Tm_{u_{11}},\ldots,w_0^Tm_{u_{1k}},w_0^Tm_{u_{21}},\ldots,w_0^Tm_{u_{2k}},\ldots,
w_0^Tm_{u_{t1}},\ldots,w_0^Tm_{u_{tk}} ).
$$
Here, for each fixed $r=1,\ldots, t$, one has
$w_0^Tm_{u_{r,h+1}}=\lambda w_0^Tm_{u_{r,h}}$, where
$h=1,\ldots,k-1$ and $t:=n/k$ is an integer. Hence, each $w_0\in
W'$ corresponds to the index vector
$P(w_0):=(u_{11},\ldots,u_{1k},\ldots, u_{t1},\ldots, u_{tk})$,
which is a permutation of $(1,\ldots,n)$.

Let us denote the set consisting of all permutations of
$(1,\ldots,n)$ by ${\bf P}$. For each $p\in {\bf P}$ put
$$
W'(p):=\{w_0\in W'|P(w_0)=p\}.
$$
We claim that there exists a $p_0\in {\bf P}$ such that ${\rm
span}(W'(p_0))={\bf R}^s$, i.e., that there are $s$ linearly
independent vectors in $W'(p_0)$. For a contradiction, assume that
${\rm span}(W'(p_0))$ is contained in a $(s-1)$-hyperplane for every
$p_0\in {\bf P}$. Note that
$$
\cup_{p\in {\bf P}}W'(p)=W'={\bf Z}^s\setminus (\cup_{j=1}^nH_j).
$$
Then
\begin{equation}\label{eq:fin}
{\bf Z}^s=(\cup_{p\in {\bf P}}W'(p))\cup(\cup_{j=1}^nH_j).
\end{equation}
 Since
$\# {\bf P}$ is finite, the equation (\ref{eq:fin}) shows that ${\bf
Z}^s$ can be written as a finite union of  hyperplanes, yielding a
contradiction.

Without loss of generality, we may suppose that ${\rm
span}(W'(p_0))={\bf R}^s$ for $p_0=(u_{11},\ldots,u_{1k},\ldots,
u_{t1},\ldots, u_{tk})$. Then, one can select $s$ linearly
independent vectors in $W'(p_0)$, say $w_1,\ldots,w_s$, such that
for each fixed $1\leq r\leq t$ and $1 \leq j\leq s$, one has
$w_j^Tm_{u_{r,h+1}}=\lambda w_j^Tm_{u_{r,h}}$, where $1\leq h\leq
k-1$. Hence, for fixed $r$ and $h$, we obtain the following linear
equations
$$w_j^Tm_{u_{r,h+1}}=w_j^T\lambda m_{u_{r,h}},$$
where $j=1,\ldots,s$. Solving these linear equations, we get
$m_{u_{r,h+1}}=\lambda m_{u_{r,h}}$,
 where $r=1,\ldots, t$ and $h=1,\ldots,k-1$.

Hence the matrices $(m_{u_{r1}},\ldots, m_{u_{rk}}), 1\leq r\leq
t,$ are of the form $(m_0,\lambda m_0,\ldots,$ $\lambda^{k-1}
m_0)$ with $m_0:=m_{u_{r1}}$. Consequently, the matrix $M$ can be
written as a union of $t$ matrices of the same form.\end{proof}

\medskip

{\bf Acknowledgements.} We thank both referees whose remarks
improved the readability of the paper. The research of the first
named author was supported by the Lithuanian Foundation of Studies
and Science. The second named author was supported by the National
Natural Science Foundation of China (10401021).




\begin{thebibliography}{00}

\bibitem{ausch} P. Auscher, Wavelet bases for $L^2(R)$ with
rational dilation factor, in: Wavelets and Their Applications,
edited by M. B. Ruskai et al. (Jones and Bartlett, 1992), pp.
439-452.

\bibitem{cdm} A. Cavaretta, W. Dahmen and C. A. Micchelli,
Stationary subdivision, {\it Mem. Amer. Math. Soc.,} {\bf 93}
(1991), 1-186.

\bibitem{daispline} X.-R. Dai, D.-J. Feng and Y. Wang,
Classification of refinable splines, {\it  Constructive Approx.},
{\bf 24} (2006), 187-200.

\bibitem{dai} X.-R. Dai, D.-J. Feng and Y. Wang, Refinable
functions with non-integer dilations, {\it J. Func. Anal.,} {\bf
250} (2007), 1-20.

\bibitem{dai2} X.-R. Dai, D.-J. Feng and Y. Wang,
Structure of refinable splines, {\it Appl. Comput. Harmonic
Anal.,} {\bf 22} (2007), 374-381.

\bibitem{D1} I. Daubechies and J. C. Lagarias, Two-scale
difference equations I. Existence and global regularity of
solutions, {\it SIAM J. Math. Anal.,} {\bf  22} (1991), 1388-1410.

\bibitem{D2} I. Daubechies, Ten lectures on wavelets, CBMS-NSF
Regional Conference Series in Applied Mathematics, 61, Society for
Industrial and Applied Mathematics (SIAM), Philadelphia, PA, 1992.

\bibitem{dem}
B. de Mathan, Numbers contravening a condition in density modulo
$1,$ {\it Acta Math. Acad. Sci. Hung.,} {\bf 36} (1980), 237-241.

\bibitem{deboorbook} C. de Boor, K. H\"{o}llig and S. Riemenschneider, {\it Box
Splines}, Springer-Verlag, New York, 1993.

\bibitem{devore} R. DeVore and A. Ron, Developing a
computation-friendly mathematical foundation for spline functions,
SIAM News, May (2005) p.5.

\bibitem{scan}
A. Dubickas, On the fractional parts of lacunary sequences, {\it
Math. Scand.,} {\bf 99} (2006), 136-146.

\bibitem{b1} P. Erd\"os, On a family of symmetric Bernoulli
convolutions, {\it Amer. J. Math.}, {\bf 61} (1939), 974-976.

\bibitem{erd}
P. Erd\"os, Problems and results on Diophantine approximations.
II, {\it Repartition modulo 1, Actes Colloq. Marseille-Luminy
1974, Lecture Notes in Math.,} {\bf 475} (1975), 89-99.

\bibitem{b2} D. J. Feng and Y. Wang, Bernoulli convolutions
associate with certain non-Pisot numbers, {\it Adv. Math.}, {\bf
187} (2004), 173-194.

\bibitem{goodman} T. N. T. Goodman, Refinable spline functions,
in: C.C. Chui,
L.L. Schumaker (Eds.), Approximation Theory IX, Vanderbilt
University Press, Nashville. TN. 1998, pp. 1-25.

\bibitem{guan} Y. Guan, S. Lu and Y. Tang,
Characterization of compactly supported refinable splines whose
shifts form a Riesz basis, {\it J. Appr. Th.}, 133 (2005), 245-250.

\bibitem{jia} R. Q. Jia and C. A. Micchelli, Using the refinement equations
for the construction of pre-wavelets. II. Powers of two, Curves
and Surfaces, 209-246, Academic Press, Boston, MA, 1991.

\bibitem{jia1} R. Q. Jia and N. Sivakumar, On the linear
independence of integer translates of box splines with rational
direction, {\it Linear Algebra and Appl.,} {\bf 135} (1990),
19-31.

\bibitem{khi}
A.~Khintchine,  \"Uber eine Klasse linearer diophantischer
Approximationen, {\it Rend. Circ. Mat. Palermo,} {\bf 50} (1926),
170-195.

\bibitem{lang} S. Lang, {\it Algebra,} 3rd ed., Graduate texts in
mathematics {\bf 211}, Springer--Verlag, New York, Berlin, 2002.

\bibitem{lawton} W. Lawton, S. L. Lee and Z. Shen, Characterization of compactly
supported refinable splines, {\it Adv. Comp. Math.,} {\bf 3} (1995),
137-145.


\bibitem{b3} Y. Peres and W. Schlag, Smoothness of projections,
Bernoulli convolutions, and the dimension of exceptions, {\it Duke
Math. J.}, {\bf 102} (2000), 193-251.


\bibitem{pol1}
A. D. Pollington, On the density of the sequence $\{n_k \xi\}$,
{\it Illinois J. Math.,} {\bf 23} (1979), 511-515.

\bibitem{schi} A. Schinzel, {\it Polynomials with special regard to
reducibility,} Encyclopedia of mathematics and its applications
{\bf 77}, CUP, Cambridge, 2000.


\bibitem{b4} B. Solomyak, On the random series $\sum \pm \lambda^n$
(an Erd\"os problem), {\it Ann.  Math.}, {\bf 142} (1995), 611-625.

\bibitem{sun} Q. Sun, Refinable
functions with compact support, {\it J. Appr. Th.,} {\bf 86} (1996),
240-252.


\bibitem{zhou} D. X. Zhou, Some characterizations for
box spline wavelets and linear Diophantine equations, {\it Rocky
Mountain J. Math.}, {\bf 28} (1998), 1539-1560.


\end{thebibliography}
\end{document}